\documentclass[11pt,reqno]{amsart}

\usepackage[utf8]{inputenc}
\usepackage[margin=1.1in]{geometry}
\usepackage{amsmath,amssymb}
\usepackage{enumitem}
\usepackage{microtype}
\usepackage[
  hidelinks,
  pdftitle={Mapping class groups of negatively curved four-manifolds}
]{hyperref}

\setlist{nosep,leftmargin=2em}

\newtheorem{theorem}{Theorem}[section]
\newtheorem{proposition}[theorem]{Proposition}
\newtheorem{lemma}[theorem]{Lemma}
\newtheorem{corollary}[theorem]{Corollary}

\newcommand{\Diff}{\operatorname{Diff}}
\newcommand{\Homeo}{\operatorname{Homeo}}
\newcommand{\Out}{\operatorname{Out}}
\newcommand{\Int}{\operatorname{int}}
\newcommand{\id}{\operatorname{id}}
\newcommand{\ABG}{A_{\mathrm{BG}}}
\newcommand{\R}{\mathbb R}
\newcommand{\Z}{\mathbb Z}
\newcommand{\Q}{\mathbb Q}
\newcommand{\TTop}{T_{\mathrm{TOP}}}
\newcommand{\Ic}{I_c}
\newcommand{\Lc}{L_c}

\title{Mapping class groups of negatively curved four-manifolds}
\author{Weizhe Niu}
\date{}
\address{Yau Mathematical Sciences Center, Tsinghua University}
\email{weizheniu@mail.tsinghua.edu.cn}
\subjclass[2020]{Primary 57R50; Secondary 57S05, 53C20}
\keywords{mapping class groups, negatively curved manifolds, Torelli groups, Budney--Gabai diffeomorphisms, pseudo-isotopy}

\begin{document}

\begin{abstract}
For every closed oriented smooth four-manifold $M$ admitting a metric of strictly negative sectional curvature, we prove that
$\dim_{\Q}H_1(T(M);\Q)=\dim_{\Q}H_1(\TTop(M);\Q)=\infty$ where $T(M)$ and
$T_{\mathrm{TOP}}(M)$ denote the Torelli groups.
This extends a theorem of Budney--Gabai for real hyperbolic manifolds. 
\end{abstract}

\maketitle

\section{Introduction}

For a closed oriented smooth manifold $M$, we denote the Torelli groups by
\[
  \begin{aligned}
    T(M)&=\ker\!\left[\pi_0\Diff^+(M)\longrightarrow\Out(\pi_1M)\right],\\
    \TTop(M)&=\ker\!\left[\pi_0\Homeo^+(M)\longrightarrow\Out(\pi_1M)\right]
  \end{aligned}.
\]

\begin{theorem}\label{thm:main}
Let $M$ be a closed oriented smooth four-manifold admitting a Riemannian metric of strictly negative sectional curvature.  Then
\[
  \dim_{\Q}H_1(T(M);\Q)
  =\dim_{\Q}H_1(\TTop(M);\Q)=\infty.
\]
In particular, neither Torelli group is finitely generated.  More precisely, there is an injective homomorphism
\[
  \bigoplus_{k\geq 4}\Z\hookrightarrow T(M)
\]
whose image may be represented by diffeomorphisms that
\begin{enumerate}[label=\textup{(\roman*)}]
  \item are homotopic to the identity;
  \item are smoothly pseudo-isotopic to the identity.
\end{enumerate}
Moreover, the composition
\[
  \bigoplus_{k\geq 4}\Z\longrightarrow \pi_0\Diff^+(M)
  \longrightarrow \pi_0\Homeo^+(M)
\]
is injective.
\end{theorem}

Budney--Gabai prove that their implanted barbell mapping classes remain independent on closed real hyperbolic four-manifolds \cite[Theorem~5.1 and its proof]{BG}. Their construction also implies both Torelli-$H_1$ conclusions in that case \cite[Section~4 and the paragraph following the proof of Theorem~5.1]{BG}. Theorem~\ref{thm:main} extends these conclusions to every closed four-manifold admitting a metric of strictly negative sectional curvature.

Our proof adapts Budney--Gabai's argument to variable negative curvature, using the compactification of Farrell--Ontaneda \cite{FO}.

By a theorem of Paulin \cite{Paulin}, $\Out(\pi_1M)$ is finite.
Hence $T(M)$ and $\TTop(M)$ have finite index in
$\pi_0\Diff^+(M)$ and $\pi_0\Homeo^+(M)$, respectively.
A finite-index subgroup of a finitely generated group is finitely
generated, so the full mapping class groups are also not finitely generated.

\begin{corollary}\label{cor:full-mcg}
Let $M$ be as in Theorem~\ref{thm:main}. Then $\pi_0\Diff^+(M)$ and
$\pi_0\Homeo^+(M)$ are not finitely generated.
\end{corollary}

\begin{corollary}\label{cor:ch}
If $X$ is a closed complex hyperbolic surface, then
\[
  \dim_{\Q}H_1(T(X);\Q)
  =
  \dim_{\Q}H_1(\TTop(X);\Q)
  =
  \infty.
\]
In particular, this holds for every fake projective plane.
\end{corollary}

\begin{proof}
With the normalization in which complex hyperbolic space has holomorphic sectional curvature $-1$, its real sectional curvatures lie in $[-1,-1/4]$ \cite{Goldman}.  Thus Theorem~\ref{thm:main} applies. A fake projective plane is a compact complex-hyperbolic surface \cite{PrasadYeung}.
\end{proof}

\medskip
\par\noindent\textbf{Acknowledgement of AI use.}
The author used AI for proofreading, LaTeX formatting and English-language editing.

\section{Budney--Gabai's mapping classes}

For a compact manifold $V$, we use
$\Diff(V\mathbin{\mathrm{rel}}\partial V)$ and
$\Homeo(V\mathbin{\mathrm{rel}}\partial V)$
to denote the automorphism groups fixing $\partial V$ pointwise.
When extending a relative map, isotopy, or pseudo-isotopy by the identity, we use a collar-fixing representative.

For a boundary-fixing diffeomorphism, the \emph{preferred lift} to a cyclic cover is the unique lift fixing a chosen lift of a boundary basepoint.

\begin{theorem}[Budney--Gabai and Kosanovi\'c]\label{thm:local-input}
There are diffeomorphisms
$\delta_k\in\Diff(S^1\times D^3\mathbin{\mathrm{rel}}\partial)$, $k\geq4$, with the following properties.
\begin{enumerate}[label=\textup{(\roman*)}]
  \item Their mapping classes freely generate $\ABG\cong\bigoplus_{k\geq4}\Z$ in
  $\pi_0\Diff(S^1\times D^3\mathbin{\mathrm{rel}}\partial)$, and the forgetful map embeds $\ABG$ in \(
  \pi_0\Homeo(S^1\times D^3\mathbin{\mathrm{rel}}\partial).
\)
  \item For every $k\geq 4$ there is a positive integer $d_k$ such that the preferred lift of $\delta_k$ to the $d_k$-fold cyclic cover of the $S^1$-factor is smoothly isotopic to the identity relative to the boundary.
  \item Every element of $\ABG$ is smoothly pseudo-isotopic to the identity relative to the boundary.
\end{enumerate}
\end{theorem}

Items~(i) and~(ii) are due to Budney--Gabai
\cite[Theorem~3.1, Proposition~4.1, and the proof of Theorem~5.1]{BG}.
Item~(iii) follows from Kosanovi\'c
\cite[Corollary~7.4 and Proposition~4.4]{Kosanovic}.

\section{Geometry of the cyclic cover}

Let $M$ be as in Theorem~\ref{thm:main}. Let $\widetilde M$ be its universal cover, and write $G=\pi_1(M)$. Since $M$ is compact, its sectional curvatures satisfy $-b^2\leq K\leq-a^2<0$ for some $a,b>0$. The lifted metric makes $\widetilde M$ a Hadamard manifold: it is
complete, simply connected, and has nonpositive sectional curvature.

\subsection{A primitive geodesic}

\begin{lemma}\label{lem:systole}
There is a shortest nontrivial embedded closed geodesic $c\subset M$. It is primitive and its oriented normal bundle is trivial. In particular, a sufficiently small closed tube has an oriented product framing $U=\nu(c)\cong S^1\times D^3$.
\end{lemma}

\begin{proof}
A shortest nontrivial closed geodesic \(c\) exists by compactness of the unit tangent bundle. Let $c$ have length $L$ and represent $g\in G$. Recall that every
nontrivial element of $G$ acts on $\widetilde M$ by translation along
a unique invariant geodesic line, called its \textit{axis}. If $g=h^m$ with $|m|>1$, then $h$ and $g$ have the same axis and $\ell(g)=|m|\ell(h)$, so the closed geodesic representing $h$ has length $L/|m|<L$. Thus $g$ is primitive.

A self-intersection of $c$ would split it into two based loops of length less than $L$. At least one is nontrivial, and its geodesic representative would contradict the minimality of $c$. Finally, every oriented rank-three bundle over $S^1$ is trivial, so $\nu(c)$ is trivial.
\end{proof}

Fix an orientation of $c$ and an oriented framing $U=\nu(c)\cong S^1\times D^3$. Let $g\in G$ be the primitive element represented by $c$, let $A\subset\widetilde M$ be its axis, and put $H=\langle g\rangle$ and $M_H=\widetilde M/H$. Then $g$ acts on $A$ by translation of distance $L$.

\begin{lemma}\label{lem:normal-exp}
The normal exponential map is an $H$-equivariant diffeomorphism
\[
  \exp^\perp:\nu(A)\xrightarrow{\cong}\widetilde M.
\]
\end{lemma}

\begin{proof}
We apply Farrell--Ontaneda \cite[Section~2, item~4]{FO} with $Q=M_H$ and $S=A/H$. The manifold $Q$ is complete and pinched negatively curved, and $S$ is a compact embedded totally geodesic circle whose inclusion induces an isomorphism on fundamental groups. Their lifted normal-exponential diffeomorphism is the stated map. It is $H$-equivariant by naturality of the exponential map under isometries.
\end{proof}

\subsection{Radial and visual compactification}

In the same setup, Farrell--Ontaneda \cite[Section~2, Lemmas~2.3--2.4 and the nearby construction]{FO} compactify $M_H$ by the closed normal disk bundle of $S=A/H$. Identifying $S$ with $c$, the triviality of $\nu c$ gives
\begin{equation}\label{eq:compactification}
  \overline M_H\cong D(\nu c)\cong S^1\times D^3,
\end{equation}
with $\Int\overline M_H=M_H$. Here \(D(\nu c)\) is the closed disk bundle. Choose a fiberwise orthogonal trivialization compatible with the fixed framing of $U$. For some $0<r<1$, it identifies the image of the lift of $U$ about $A$ with $S^1\times rD^3$.

Let $a^+,a^-\in\partial_\infty\widetilde M$ be the endpoints of $A$. The lifted framing identifies the unit normal bundle of $A$ with $\R\times S^2$.

\begin{lemma}\label{lem:radial-visual}
Normal geodesic rays define an $H$-equivariant homeomorphism
\[
  \Theta:\R\times S^2\xrightarrow{\cong}
  \partial_\infty\widetilde M\setminus\{a^+,a^-\}.
\]
In particular, the radial boundary of $\R\times D^3$ is equivariantly the visual boundary with the two axis endpoints removed.
\end{lemma}

\begin{proof}
Farrell--Ontaneda \cite[Section~2, item~9]{FO}, applied to $Q=M_H$ and $S=A/H$, identify the lifted unit normal sphere bundle with the visual boundary complementary to $\partial_\infty A=\{a^+,a^-\}$ by normal geodesic endpoints. Their construction also extends the identity on $\widetilde M$ to a homeomorphism
\[
  \R\times D^3\xrightarrow{\cong}
  \widetilde M\cup
  \bigl(\partial_\infty\widetilde M\setminus\{a^+,a^-\}\bigr),
\]
with the radial-normal topology on the left and the partial visual topology on the right. Both identifications are $H$-equivariant because isometries carry normal geodesic rays to normal geodesic rays.
\end{proof}

\begin{lemma}\label{lem:bounded-extension}
Let $F_s:\widetilde M\to\widetilde M$, $s\in I$, be a continuous family of homeomorphisms for which $d(F_s(x),x)\leq C$ for every $x$ and $s$, with one constant $C$. Then the family extends continuously to the radial compactification and fixes its boundary pointwise. If the family is $H$-equivariant, the extensions descend to the $H$-quotient.
\end{lemma}

\begin{proof}
If $x_i\to\xi\in\partial_\infty\widetilde M$ and $s_i\to s$, the points $F_{s_i}(x_i)$ remain within distance $C$ of $x_i$, and hence converge to the same point $\xi$ in the visual compactification. The common bound gives joint continuity in the parameter. The inverse family has the same bound, because for $y=F_s(x)$, $d(F_s^{-1}(y),y)=d(x,F_s(x))\leq C$. Thus the extensions are homeomorphisms and form an isotopy relative to the visual boundary. Lemma~\ref{lem:radial-visual} transports this extension to the radial boundary, and equivariance allows passage to the $H$-quotient.
\end{proof}

\subsection{The lifted tube and uniform cancellation}

\begin{lemma}\label{lem:malnormal}
The primitive cyclic subgroup $H=\langle g\rangle$ is malnormal:
\[
  H\cap xHx^{-1}\neq1
  \quad\Longrightarrow\quad
  x\in H.
\]
Therefore,
\(
  N_G(H)=H
  \text{ and }
  Z(G)=1.
\)
\end{lemma}

\begin{proof}
If $g^m=xg^nx^{-1}\neq1$, uniqueness of axes in strict negative curvature gives $xA=A$. The stabilizer of $A$ acts faithfully and discretely on $A\cong\R$. It contains no reflection, since a reflection of $\R$ has a fixed point and the deck action is free. It is therefore a cyclic group of translations. Since $g$ is primitive in $G$, it generates this stabilizer, and hence $x\in H$. This proves malnormality and also $N_G(H)=H$.

Now suppose that $1\neq z\in Z(G)$. Since $z$ commutes with $g$, we have $z\in N_G(H)=H$. For every $x\in G$, $z=xzx^{-1}\in H\cap xHx^{-1}$, so malnormality gives $x\in H$. Thus $G=H$, contradicting compactness of $M$ and noncompactness of $M_H=\Int\overline M_H$.
\end{proof}

Let $\widetilde U\subset\widetilde M$ be the lift of $U$ around $A$. The components of the inverse image of $U$ in $M_H=\widetilde M/H$ are indexed by double cosets $H\backslash G/H$. For a representative $x\in G$, the corresponding component is
\begin{equation}\label{eq:double-coset}
  C_x\cong x\widetilde U/(H\cap xHx^{-1}).
\end{equation}
For the identity double coset, $C_1=\widetilde U/H\cong U$. It is the compact \emph{core component} and maps degree one to $U$. If $x\notin H$, Lemma~\ref{lem:malnormal} makes the denominator in \eqref{eq:double-coset} trivial, so $C_x\cong\R\times D^3$, and $C_x\to U$ is the universal cyclic cover. Self-normalization gives the uniqueness of the compact core component.

The family of components is locally finite: each sheet over a sufficiently small evenly covered neighborhood in $M$ meets at most one component of the inverse image of $U$. Each component inherits the framing of $U$, and two covering parametrizations of a noncore component differ only by a deck translation in the $\R$-coordinate.

\begin{lemma}\label{lem:cancellation}
Let $\delta=\prod_{k=4}^{N}\delta_k^{n_k}$ be a finite word, let $F\in\Diff^+(M)$ be its implantation in $U$, and let $F_H$ be the componentwise lift to $M_H$. There is a smooth isotopy on $M_H$ from $F_H$ to a diffeomorphism $F_{\mathrm{core}}$ which equals $\delta$ on the core component and the identity away from the core. The isotopy is stationary on a neighborhood of the boundary of every lifted tube, and it has an $H$-equivariant lift to $\widetilde M$ that stays a uniformly bounded distance from the identity, with one bound for all noncore components.
\end{lemma}

\begin{proof}
The identity word is immediate, so suppose that some $n_k$ is nonzero. For each such $k$, choose the degree $d_k$ supplied by Theorem~\ref{thm:local-input}(ii), and set $d=\operatorname{lcm}\{d_k\mid n_k\neq0\}$. The $d$-fold cyclic cover $U_d\to U$ factors through every relevant $d_k$-fold cover. Pulling up the individual boundary-relative isotopies, taking inverses for negative powers, repeating them for positive powers, and concatenating gives a smooth isotopy
\[
  \Phi_s\in\Diff(U_d\mathbin{\mathrm{rel}}\partial U_d),
  \qquad
  \Phi_0=\delta^{(d)},
  \quad
  \Phi_1=\id,
\]
where $\delta^{(d)}$ is the preferred lift of $\delta$.

Equip $U_d$ with the lifted metric. Compactness gives a finite common length bound
\[
  L_0=\sup_{z\in U_d}\int_0^1
  \left\lVert\frac{\partial}{\partial s}\Phi_s(z)\right\rVert ds
  <\infty.
\]
Lift $\Phi_s$ to the universal cyclic cover $\R\times D^3$, starting with the preferred lift of $\delta$. The lifted isotopy ends at the identity. Every lifted track has length at most $L_0$.

Transport the lifted isotopy to every noncore component through its covering parametrization over the framed tube $U$. Different parametrizations differ by translations, so the transported isotopy and its bound are unchanged. Leave the core component fixed and use the identity outside the inverse image of $U$. The components are disjoint and locally finite, and all formulas are stationary near their boundaries. Thus they assemble to a smooth isotopy on $M_H$. Its endpoint is $F_{\mathrm{core}}$.

Lift the assembled isotopy to $\widetilde M$ from the $H$-equivariant lift of $F_H$. Uniqueness of lifted homotopies makes the whole family $H$-equivariant. The same number $L_0$ controls every noncore track. On the core, the isotopy is constant at the preferred lift of $\delta$.
Its displacement is bounded because the lift is equivariant under deck
translations and $U$ is compact. Thus the isotopy has uniformly bounded displacement.
\end{proof}

\section{The Torelli detector}

Since $\widetilde M$ is Hadamard, it is diffeomorphic to $\R^4$,
so $M$ is a $K(G,1)$. Fix $q\in\partial U$. Since $\pi_1(\partial U,q)\to\pi_1(U,q)$ is an isomorphism, a relative self-map of $U$ induces the identity on $\pi_1(U,q)$. Therefore extension by the identity induces the identity on $G$. Implantation gives homomorphisms $\ABG\longrightarrow T(M)\longrightarrow\TTop(M)$. Denote their composite by $\Ic:\ABG\longrightarrow\TTop(M)$.

\begin{proposition}[Torelli cyclic-cover detector]\label{prop:detector}
There is a well-defined homomorphism
\[
  \Lc:\TTop(M)\longrightarrow \pi_0\Homeo(S^1\times D^3\mathbin{\mathrm{rel}}\partial),
  \qquad
  \Lc\circ\Ic=\id_{\ABG},
\]
where $\ABG$ is identified with its image in $\pi_0\Homeo(S^1\times D^3\mathbin{\mathrm{rel}}\partial)$.
\end{proposition}

\begin{proof}
Let $[f]\in\TTop(M)$. Since $f$ induces the trivial element of
$\Out(G)$, its induced automorphism of $G$ is inner. Self-maps of the
$K(G,1)$ space $M$ are classified up to free homotopy by conjugacy
classes of their induced homomorphisms on $G$, so $f$ is freely
homotopic to the identity. Lift it with $\widetilde I_0=\id_{\widetilde M}$. For every $\gamma\in G$, the maps $\widetilde I_t\gamma$ and $\gamma\widetilde I_t$ lift the same homotopy and agree at $t=0$. Hence they agree for all $t$. The endpoint $\widetilde f$ is therefore $G$-equivariant and descends to a homeomorphism $f_H$ of $M_H$.

This endpoint lift is independent of the chosen free homotopy. Indeed, two such endpoints differ by a deck transformation $z\in G$, and equivariance makes $z$ central. Lemma~\ref{lem:malnormal} therefore gives $z=1$.

The displacement function $x\longmapsto d(x,\widetilde f(x))$ is $G$-invariant and hence bounded. Lemma~\ref{lem:bounded-extension}, applied to the constant family, extends $f_H$ to a homeomorphism $\overline f_H$ of $\overline M_H$ fixing its boundary pointwise. If two representatives of $[f]$ are isotopic, the isotopy lifts from the canonical initial lift to a $G$-equivariant family ending at the canonical terminal lift. Its displacement descends to a continuous function on the compact space $M\times I$ and is therefore uniformly bounded. Lemma~\ref{lem:bounded-extension} extends the family relative to the boundary. Thus, using the identification \eqref{eq:compactification}, $\Lc([f])=[\overline f_H]\in \pi_0\Homeo(S^1\times D^3\mathbin{\mathrm{rel}}\partial)$ is well defined. The product of the canonical lifts of $f$ and $h$ is the canonical lift of $fh$, so $\Lc$ is a homomorphism.

Let $\delta=\prod_{k=4}^{N}\delta_k^{n_k}$ be a finite word, and let $F$ be its implantation in $M$. Since $F$ fixes $q$ and induces the identity on $G$, its lift fixing a chosen lift of $q$ is $G$-equivariant and hence is the canonical lift above. Its descent to $M_H$ is the componentwise map consisting of the degree-one copy of $\delta$ on the core, the preferred cyclic lift on each noncore component, and the identity elsewhere.

Lemma~\ref{lem:cancellation} cancels all noncore copies through an isotopy whose lift to $\widetilde M$ has uniformly bounded displacement. Lemma~\ref{lem:bounded-extension} extends this to an isotopy of $\overline M_H$ relative to its boundary. Its endpoint is supported on the core $S^1\times rD^3$ and is the corresponding copy of $\delta$ there. Radial expansion gives a relative isotopy from this core-supported map to $\delta$ on $S^1\times D^3$. Hence $\Lc(\Ic([\delta]))=[\delta]$, which proves the proposition.
\end{proof}

\begin{proof}[Proof of Theorem~\ref{thm:main}]
Proposition~\ref{prop:detector} shows that $\Ic$ is injective. Consequently, the smooth implantation $\ABG\longrightarrow T(M)$ is injective and remains injective after passing to $\pi_0\Homeo^+(M)$.

Since $\pi_0\Homeo(S^1\times D^3\mathbin{\mathrm{rel}}\partial)$ is abelian, $\Lc$ factors through $\TTop(M)_{\mathrm{ab}}$. The identity $\Lc\circ\Ic=\id_{\ABG}$ therefore gives an injection
\[
  \ABG\hookrightarrow\TTop(M)_{\mathrm{ab}}.
\]
Since $\Q$ is flat over $\Z$ and $H_1(\Gamma;\Q)\cong\Gamma_{\mathrm{ab}}\otimes\Q$, we have
\[
  \bigoplus_{k\geq4}\Q
  \cong
  \ABG\otimes\Q
  \hookrightarrow
  H_1(\TTop(M);\Q).
\]
Thus $\dim_{\Q}H_1(\TTop(M);\Q)=\infty$.

The composite $T(M)\longrightarrow\TTop(M)\xrightarrow{\Lc}\pi_0\Homeo(S^1\times D^3\mathbin{\mathrm{rel}}\partial)$ has the same restriction to the implanted subgroup, so the same argument gives $\dim_{\Q}H_1(T(M);\Q)=\infty$.

The construction preceding Proposition~\ref{prop:detector} gives representatives inducing the identity on $\pi_1(M)$ and homotopic to the identity. Finally, Theorem~\ref{thm:local-input}(iii) gives every element of $\ABG$ a smooth pseudo-isotopy to the identity relative to $\partial U$. Extending by the identity on $(M\setminus\Int U)\times I$ gives the asserted smooth pseudo-isotopy on $M$.
\end{proof}


\begin{thebibliography}{99}

\bibitem{BG}
R.~Budney and D.~Gabai,
\emph{On the automorphism groups of hyperbolic manifolds},
Int. Math. Res. Not. IMRN 2025, no.~7, rnaf083,
\href{https://doi.org/10.1093/imrn/rnaf083}{doi:10.1093/imrn/rnaf083}.

\bibitem{FO}
F.~T.~Farrell and P.~Ontaneda,
\emph{On the topology of the space of negatively curved metrics},
J. Differential Geom. \textbf{86} (2010), 273--301.

\bibitem{Goldman}
W.~M.~Goldman,
\emph{Complex hyperbolic geometry},
Oxford Mathematical Monographs, The Clarendon Press, Oxford University Press, New York, 1999.

\bibitem{Kosanovic}
D.~Kosanovi\'c,
\emph{Diffeomorphisms of $4$-manifolds from graspers},
Proc. Lond. Math. Soc. (3) \textbf{131} (2025), no.~1, e70065,
\href{https://doi.org/10.1112/plms.70065}{doi:10.1112/plms.70065}.

\bibitem{Paulin}
F.~Paulin,
\emph{Outer automorphisms of hyperbolic groups and small actions on
$\mathbb R$-trees},
in \emph{Arboreal group theory} (Berkeley, CA, 1988),
Math. Sci. Res. Inst. Publ., vol.~19,
Springer, New York, 1991, pp.~331--343.

\bibitem{PrasadYeung}
G.~Prasad and S.-K.~Yeung,
\emph{Fake projective planes},
Invent. Math. \textbf{168} (2007), no.~2, 321--370;
addendum, Invent. Math. \textbf{182} (2010), no.~1, 213--227.

\end{thebibliography}
\end{document}